\documentclass[a4paper,12pt]{article}
\usepackage{amsmath, amscd, amsfonts, amsthm, amssymb, latexsym, comment, stmaryrd, graphicx}
\usepackage[arrow, matrix, curve]{xy}
\usepackage{hyperref}
\usepackage{amsmath}
\usepackage{geometry}

\newtheorem{leer}{}[section]
\newtheorem{thm}[leer]{Theorem}

\newtheorem{prop}[leer]{Proposition}
\newtheorem{lemm}[leer]{Lemma}
\newtheorem{coro}[leer]{Corollary}

\theoremstyle{remark}
\newtheorem{rema}[leer]{Remark}

\newcommand{\Spec}{\mathrm{Spec}}
\newcommand{\Gal}{\mathrm{Gal}}

\def\Aa{{\mathbb{A}}}

\begin{document}
\parindent0em
\parskip1em

\title{On varieties of Hilbert type}
\author{Lior Bary-Soroker, Arno Fehm and Sebastian Petersen}
\maketitle

\begin{abstract}
A variety $X$ over a field $K$ is of Hilbert type if $X(K)$ is not thin.
We prove that if $f\colon X\rightarrow S$ is a dominant morphism of $K$-varieties and both $S$ and all fibers $f^{-1}(s)$, $s\in S(K)$, are of Hilbert type, then so is $X$.
We apply this to answer a question of Serre on products of varieties and to generalize a result of Colliot-Th\'el\`ene and Sansuc on algebraic groups.
\end{abstract}

\section{Introduction}
In the terminology of thin sets (we recall this notion in Section~\ref{sec:notation}), Hilbert's irreducibility theorem asserts that $\mathbb{A}_K^n(K)$ is not thin,
for any number field $K$ and any $n\geq 1$. 
As a natural generalization a $K$-variety $X$ is called of Hilbert type if $X(K)$ is not thin. The importance of this definition stems from the observation of Colliot-Th\'el\`ene and Sansuc \cite{CTS} that the inverse Galois problem would be settled if every unirational variety over $\mathbb{Q}$ was of Hilbert type.

In this direction, Colliot-Th\'el\`ene and Sansuc \cite[Corollary 7.15]{CTS} prove that any connected reductive algebraic group over a number field is of Hilbert type. This immediately raises the question whether the same holds for all linear algebraic groups (note that these are unirational).
Another question, asked by Serre \cite[p.~21]{Serre}, is whether a product of two varieties of Hilbert type is again of Hilbert type.
The main result of this paper gives a sufficient condition for a variety to be of Hilbert type:

\begin{thm}\label{mainthm}
Let $K$ be a field and $f\colon X\to S$ a dominant morphism of $K$-varieties.
Assume that the set of $s\in S(K)$ for which the fiber $f^{-1}(s)$ is a $K$-variety of Hilbert type is not thin. 
Then $X$ is of Hilbert type.
\end{thm}

As an immediate consequence we get the following result
for a family of varieties over a variety of Hilbert type:
 
\begin{coro}\label{maincor}
Let $K$ be a field and $f\colon X\to S$ a dominant morphism of $K$-varieties.
Assume that $S$ is of Hilbert type and that for every $s\in S(K)$ the fiber $f^{-1}(s)$ is of
Hilbert type. Then $X$ is of Hilbert type.
\end{coro}

Using this result we resolve both questions discussed above affirmatively, see Corollary~\ref{cor:Serre} and Proposition~\ref{prop:linear}.

\section{Background}
\label{sec:notation}

Let $K$ be a field. A {\em $K$-variety} is a separated scheme of finite type over $K$ which is geometrically reduced and geometrically irreducible.
Thus, an open subscheme of a $K$-variety is again a $K$-variety.
If $f\colon X\to S$ is a morphism of $K$-varieties and $s\in S(K)$, then $f^{-1}(s):=X\times_S \Spec(\kappa(s))$,
where $\kappa(s)$ is the residue field of $s$,
denotes the scheme theoretic fiber of $f$ at $s$. 
This fiber is a separated scheme of finite type over $K$, which needs not be reduced
or connected in general. We identify the set $f^{-1}(s)(K)$ of $K$-rational points of the fiber with the set theoretic fiber
$\{x\in X(K)\mid f(x)=s\}$. 

Let $X$ be a $K$-variety.
A subset $T$ of $X(K)$ is called {\em thin} 
if there exists a proper Zariski-closed subset $C$ of $X$, a finite set $I$, and for each $i\in I$ a $K$-variety $Y_i$ with $\dim Y_i = \dim X$ and a dominant separable morphism $p_i\colon Y_i\to X$ of degree $\geq 2$ such that 
\begin{eqnarray*}
T &\subseteq& \bigcup_{i\in I} p_i(Y_i(K))\cup C(K).
\end{eqnarray*}
A $K$-variety $X$ is {\em of Hilbert type} if $X(K)$ is not thin, cf.~\cite[Definition~3.1.2]{Serre}.
Note that $X$ is of Hilbert type if and only if some (or every) open subscheme of $X$ is of Hilbert type, cf.~\cite[p.~20]{Serre}.
A field $K$ is {\em Hilbertian} if $\Aa_K^1$ is of Hilbert type.
We note that if there exists a $K$-variety $X$ of positive dimension such that $X$ is of Hilbert type, then $K$ is Hilbertian 
\cite[Proposition 13.5.3]{FriedJarden}.
%\cite[p. 20, Exercise~1]{Serre}. 

\section{Proof of Theorem \ref{mainthm}}
A key tool in the proof of Theorem \ref{mainthm} is the following consequence of Stein factorization. 

\begin{lemm}\label{kol}
Let $K$ be a field and $\psi\colon Y\to S$ a dominant morphism of normal $K$-varieties. 
Then there exists a nonempty open subscheme $U\subset S$, a $K$-variety $T$ and a factorization 
\[
\xymatrix{
\psi^{-1}(U)\ar[r]^-{g}& T\ar[r]^-{r}&U
}
\]
of $\psi$ such that the fibers of $g$ are geometrically irreducible and $r$ is finite and \'etale.
\end{lemm}

\begin{proof}
See \cite[Lemma 9]{kollar}.
\end{proof}

\begin{lemm} \label{keylemm} 
Let $K$ be a field and $f\colon X\to S$ a dominant morphism of normal $K$-varieties. 
Assume that %$S$ is of Hilbert type and 
the set $\Sigma$ of $s\in S(K)$ for which $f^{-1}(s)$ is a $K$-variety of Hilbert type is not thin.
Let $I$ be a finite set and let $p_i\colon Y_i\to X$, $i\in I$, be finite \'etale morphisms of degree $\geq 2$. Then
$X(K) \not\subseteq \bigcup_{i\in I} p_i(Y_i(K))$.
\end{lemm} 

\begin{proof}
For $i\in I$ consider the
composite morphism $\psi_i:=f\circ p_i\colon Y_i\to S$. 
By Lemma~\ref{kol} there is a nonempty open subscheme $U_i$ of $S$ and a factorization 
$$\psi_i^{-1}(U_i)\buildrel g_i\over\longrightarrow T_i\buildrel r_i\over\longrightarrow U_i$$
of $\psi_i$ such that the morphism $g_i$ has geometrically irreducible fibers, $r_i$ is finite and \'etale, and
such that $T_i$ is a $K$-variety.
We now replace successively $S$ by $\bigcap_{i\in I} U_i$, $X$ by $f^{-1}(S)$, $T_i$ by $r_i^{-1}(S)$ and $Y_i$ by $p_i^{-1} (X)$,
to assume in addition that $r_i\colon T_i\rightarrow S$ is finite \'etale for every $i\in I$.

For $s\in S(K)$ denote by $X_s:=f^{-1}(s)$ the fiber of $f$ over $s$. 
Then $X_s$ is a $K$-variety of Hilbert type for each $s\in\Sigma$. 
Furthermore we define $Y_{i, s}:=\psi_i^{-1}(s)$
and let $p_{i,s}\colon Y_{i,s}\rightarrow X_s$ be the corresponding projection morphism.
Then $p_{i,s}$ is a finite \'etale morphism of the same degree as $p_i$.
In particular, the $K$-scheme $Y_{i, s}$ is geometrically reduced. 
For every $s\in S(K)$ and every $i\in I$ we have constructed a commutative diagram
\[
\xymatrix{
Y_{i, s}\ar[d]_{p_{i,s}}\ar[r] & Y_i\ar[d]_{p_i}\ar[r]^{g_i}\ar[dr]^{\psi_i} & T_i\ar[d]^{r_i}\\ 
X_s \ar[r]     &  X\ar[r]^{f} &  S
 }
 \]
in which the left hand rectangle is cartesian.
Set $J:=\{i\in I: \deg(r_i)\ge 2\}$. 
%As $S$ is of Hilbert type, we can choose a point
%$s\in S(K)$ with  
Then $\bigcup_{i\in J} r_i(T_i(K)) \subseteq S(K)$
is thin, so by assumption there exists 
$$
 s\in \Sigma\smallsetminus \bigcup_{i\in J} r_i(T_i(K)).
$$ 
For $i\in J$ there is no $K$-rational point of $T_i$ over $s$, hence $Y_{i, s}(K)=\emptyset$ for every $i\in J$. 
For $i\in I\smallsetminus J$, the finite \'etale morphism $r_i$ is of degree $1$, hence an isomorphism, and therefore $Y_{i, s}$ is geometrically irreducible.
Thus, $Y_{i,s}$ is a $K$-variety.
So since $X_s$ is of Hilbert type, there exists $x\in X_s(K)$ such
that $x\notin \bigcup_{i\in I\smallsetminus J} p_{i,s}(Y_{i, s}(K))$. Thus  
\[
x\not\in \bigcup_{i\in J\smallsetminus I}p_{i,s}(Y_{i, s}(K)) = \bigcup_{i\in I}p_{i,s}(Y_{i, s}(K)),
\]
hence $x\not\in \bigcup_{i\in I} p_i(Y_i(K))$, as needed.
\end{proof}

We shall use the following well-known fact.

\begin{lemm}\label{shrinking}
Let $K$ be a field, let $X,Y$ be $K$-varieties with $\dim X=\dim Y$,  and let $p\colon Y\to X$ be a dominant separable morphism. 
Then there exists a nonempty open subscheme $U$ of $X$ such that the restriction of $p$ to
a morphism $p^{-1}(U)\to U$ is finite and \'etale.
\end{lemm}

\begin{proof}[Proof of Theorem \ref{mainthm}]
Let $K$ be a field, and $f\colon X\to S$ a dominant morphism of $K$-varieties. 
Assume that the set $\Sigma$ of those $s\in S(K)$ for which $f^{-1}(s)$ is of Hilbert type is not thin. 
Let $C\subseteq X$ be a proper Zariski-closed subset. Let $I$ be a finite set and suppose that $Y_i$ is a $K$-variety with $\dim(Y_i)=\dim(X)$ and $p_i\colon Y_i\to X$ is a dominant separable morphism of degree $\geq 2$, for every $i\in I$.
We have to show that $X(K)\not\subseteq C(K)\cup\bigcup_{i\in I}p_i(Y_i(K))$.

By Lemma \ref{shrinking} and \cite[6.12.6, 6.13.5]{EGAIV2} there exists a {normal} 
nonempty open subscheme $X'\subset X\smallsetminus
C$ such that the restriction of each $p_i$ to a morphism $p_i^{-1}(X')\to X'$ is finite and \'etale. 
The image $f(X')$ contains
a nonempty open subscheme $S'$ of $S$ (cf. \cite[1.8.4]{EGAIV1}, \cite[9.2.2]{EGAIII1}). 
Furthermore, $S'$ contains a 
nonempty normal open subscheme $S''$. Let us define $X'':=f^{-1}(S'')\cap X'$ and $Y''_i:=p_i^{-1}(X'')$.
Then the restriction of $f$ to a 
morphism $f''\colon X''\to S''$ is a surjective morphism of normal $K$-varieties, 
$\Sigma\cap S''(K)$ is not thin,
and $f''^{-1}(s)$ is of Hilbert type for every
$s\in \Sigma\cap S''(K)$ because it is an open subscheme of $f^{-1}(s)$. 
The restriction $p_i''$ of $p_i$ to a morphism $Y_i''
\to X''$ is finite and \'etale for every 
$i\in I$. 
%By Lemma \ref{keylemm} there exists a point $x\in X''(K)\smallsetminus \bigcup_{i\in I} p_i''(Y_i(K))$. Then $x\in X(K)\smallsetminus C(K)\cup\bigcup_{i\in I}p_i(Y_i(K))$ as desired.
By Lemma \ref{keylemm} applied to $f''$ and the $p_i''$ we have 
\begin{eqnarray*}
\emptyset&\neq& X''(K)\smallsetminus \bigcup_{i\in I} p_i''(Y_i''(K)) \\
&=& X''(K)\smallsetminus \bigcup_{i\in I} p_i(Y_i(K))\\
&\subseteq& X(K)\smallsetminus \Big(C(K)\cup\bigcup_{i\in I}p_i(Y_i(K))\Big),
\end{eqnarray*}
so $X(K)\not\subseteq C(K)\cup\bigcup_{i\in I}p_i(Y_i(K))$, as needed. 
\end{proof}

%\begin{proof}[Proof of Corollary \ref{maincor}]
%If $S$ is of Hilbert type, then $S(K)$ is not thin, so if
%the set $\Sigma$ of those $s\in S(K)$ for which $f^{-1}(s)$ is of Hilbert type equals $S(K)$, then $\Sigma$ is not thin.
%Thus, the claim follows from Theorem \ref{mainthm}.
%\end{proof}

As an immediate consequence we get an affirmative solution of Serre's question mentioned in the introduction.

\begin{coro}\label{cor:Serre} Let $K$ be a field.
If $X,Y$ are $K$-varieties of Hilbert type, then $X\times Y$ is of Hilbert type.
\end{coro}

\begin{proof} Denote by $f\colon X\times Y\to X$ the projection. Then $f^{-1}(x)$ is isomorphic to $Y$ and hence of Hilbert
type for every $x\in X(K)$. Thus Corollary \ref{maincor} gives that $X\times Y$ is of Hilbert type.
\end{proof}

\section{Algebraic groups of Hilbert type}

By an {\em algebraic group} over a field $K$ we shall mean a connected smooth group scheme over $K$.
Recall that such an algebraic group is a $K$-variety, see \cite[Exp ${\rm VI_A}$, 0.3, 2.1.2, 2.4]{SGA3}.
If $G$ is an algebraic group over $K$, then $G(K_s)$ is a $\Gal(K)$-group,
where $K_s$ denotes a separable closure of $K$ and $\Gal(K)=\Gal(K_s/K)$ is the absolute Galois group of $K$,
and there is an associated Galois cohomology pointed set $H^1(K, G)=H^1(\Gal(K),G(K_s))$,
which classifies isomorphism classes of  $G(K_s)$-torsors, cf.~\cite[Prop. 1.2.4]{NSW}. %~\cite[p.~16]{NSW}. 

\begin{prop}\label{cor:exact} 
Let $K$ be a field and let 
$$
 1\rightarrow N\rightarrow G\buildrel{p}\over\longrightarrow Q \rightarrow 1
$$ 
be a short exact sequence of algebraic groups over $K$. If 
$H^1(K, N)=1$
and both $N$ and $Q$ are of Hilbert type, then $G$ is of Hilbert type.
\end{prop}

\begin{proof} 
It suffices to show that $p^{-1}(x)$ is of Hilbert type for every $x\in Q(K)$,
because then Corollary \ref{maincor} implies the assertion.
Let $x\in Q(K)$ and $F=p^{-1}(x)$. 
%Since $N$ is smooth by definition, 
There is an exact sequence of $\Gal(K)$-groups 
$$
 1\to N(K_s)\to G(K_s)\to Q(K_s)\to 1,
$$ 
where the right hand map is surjective, because for every point $x\in Q(K_s)$ the fiber over $x$ is 
a non-empty $K_s$-variety and thus has a $K_s$-rational point.
Since the $\Gal(K)$-set $F(K_s)$ is a coset of $N(K_s)$, it
is a $N(K_s)$-torsor. Our hypothesis 
$H^1(K, N)=1$ implies that $F(K_s)$ is isomorphic to the trivial $N(K_s)$-torsor $N(K_s)$. 
It follows that $F$ is isomorphic to $N$ as a $K$-variety,
hence $F$ is of Hilbert type.
\end{proof}

Using this, we generalize the result of Colliot-Th\'el\`ene and Sansuc \cite[Corollary~7.15]{CTS} from reductive groups to arbitrary linear groups. 

\begin{thm}\label{prop:linear}
Every linear algebraic group $G$ over a perfect Hilbertian field $K$ is of Hilbert type.
\end{thm}

\begin{proof} 
We denote by $G_u$ the unipotent radical of $G$ (cf. \cite[Proposition XVII.1.2]{milne}). 
We have a short exact sequence
of algebraic groups over $K$
\renewcommand{\theequation}{$\ast$}
\begin{equation}\label{eq}
 1\rightarrow G_u\rightarrow G\rightarrow Q\rightarrow 1 
\end{equation}
with $Q$ reductive, cf.~\cite[Proposition XVII.2.2]{milne}.
By \cite[Corollary 7.15]{CTS}, $Q$ is of Hilbert type. 
Since $K$ is perfect, $G_u$ is split, i.e.~there exists a series of normal algebraic subgroups
$$
 1=U_0\subseteq\dots\subseteq U_n=G_u
$$
such that $U_{i+1}/U_i\cong\mathbb{G}_a$ for each $i$,
cf. \cite[15.5(ii)]{borel}. The groups $U_i$ are unipotent, hence $H^1(K, U_i)=1$ by 
\cite[Ch.~III.2.1, Prop.~6]{SerreGalois}, and $\mathbb{G}_a$ is of Hilbert type since $K$ is Hilbertian.
Thus, an inductive application of Proposition \ref{cor:exact} 
implies that $G_u$ is of Hilbert type.
Finally we apply Proposition \ref{cor:exact} to the exact sequence (\ref{eq}) to conclude that $G$ is
of Hilbert type.
\end{proof}

\begin{rema}
The special case of Theorem \ref{prop:linear} where $G$ is simply connected and $K$ is finitely generated is also a consequence of a result of Corvaja, see \cite[Corollary 1.7]{Corvaja}.
\end{rema}

\begin{rema}
We point out that Theorem \ref{prop:linear} could be deduced also from Corollary \ref{cor:Serre} (instead of \ref{mainthm})
via \cite[Corollary 1]{Rosenlicht} and the fact that a unipotent group over a perfect field is rational, cf.~\cite[XIV, 6.3]{SGA3}.
\end{rema}

As a consequence of Theorem \ref{prop:linear}, we get a more general statement for homogeneous spaces,
which was pointed out to us by Borovoi:

\begin{coro}\label{cor:Borovoi}
If $G$ is a linear algebraic group over a perfect Hilbertian field $K$, and $H$ is algebraic subgroup of $G$, then
the quotient $G/H$ is of Hilbert type.
\end{coro}

\begin{proof}
For the existence of the quotient $Q:=G/H$ see for example \cite[Chapter II Theorem 6.8]{borel}.
If $\mathcal{H}$ denotes the generic fiber of $G\to Q$ and $\bar{F}$ is an algebraic closure of the function field $K(Q)$ of $Q$,
then $\mathcal{H}_{\bar{F}}\cong H_{\bar{F}}$ by translation on $G$.
Thus, $\mathcal{H}$ is geometrically irreducible since $H$ is,
so \cite[Proposition~7.13]{CTS} implies that $Q$ is of Hilbert type.
\end{proof}

%As a consequence of Theorem \ref{prop:linear}, 
We also get a complete classification of the algebraic groups that are of Hilbert type over a number field:

\begin{coro}
An algebraic group $G$ over a number field $K$ 
is of Hilbert type if and only if it is linear.
\end{coro}

\begin{proof}
If $G$ is linear, then it is of Hilbert type by Theorem \ref{prop:linear}.
Conversely, assume that $G$ is of Hilbert type.
Chevalley's theorem \cite[Theorem 1.1]{conrad2002} gives a short exact sequence of algebraic groups over $K$,
$$
 1 \rightarrow H\rightarrow G\rightarrow A\rightarrow 1
$$ 
with $H$ linear and $A$ an abelian variety.
%By Corollary \ref{cor:Borovoi}, the quotient $A=G/H$ is of Hilbert type.
As in the proof of Corollary \ref{cor:Borovoi} we conclude that the generic fiber of $G\rightarrow A$ is geometrically irreducible,
and therefore $A$ is of Hilbert type.
%If $\mathcal{H}$ denotes the generic fiber of $G\to A$ and $\bar{F}$ is an algebraic closure of the function field $K(A)$ of $A$,
%then $\mathcal{H}_{\bar{F}}\cong H_{\bar{F}}$ by translation on $G$.
%Thus, $\mathcal{H}$ is geometrically irreducible since $H$ is,
%so \cite[Proposition~7.13]{CTS} implies that $A$ is of Hilbert type.
Since no
nontrivial abelian variety over a number field is of Hilbert type, 
cf.~\cite[Remark 13.5.4]{FriedJarden},
%cf.~\cite[p.~20, Exercise 4]{Serre},
$A$ is trivial and $G\cong H$ is linear.
\end{proof}

\section*{Acknowledgements}

The authors would like to thank 
Mikhail Borovoi for pointing out to them Corollary \ref{cor:Borovoi},
Jean-Louis Colliot-Th\'el\`ene and Moshe Jarden for helpful comments on a previous version,
and Daniel Krashen for some references concerning algebraic groups.
This research was supported by the Lion Foundation Konstanz -- Tel Aviv, 
the Alexander von Humboldt Foundation, 
the Hermann Minkowski Minerva Center for Geometry at Tel Aviv University,
and by a grant from the GIF, the German-Israeli Foundation for Scientific Research and Development.

\end{document}